\documentclass[a4paper,reqno]{amsart}

\usepackage{amsmath}
\usepackage{paralist}
\usepackage{graphics}
\usepackage{epsfig}
\usepackage{amsfonts}
\usepackage{amssymb}
\usepackage{hyperref}
\usepackage{epstopdf}
\usepackage{color}
\usepackage{amsthm}

\newtheorem{theorem}{Theorem}[section]
\newtheorem{lemma}[theorem]{Lemma}

\def\ifl{\iffalse }

\def\bc{\begin{center}}       \def\ec{\end{center}}
\def\ba{\begin{array}}        \def\ea{\end{array}}
\def\be{\begin{equation}}     \def\ee{\end{equation}}
\def\bea{\begin{eqnarray}}    \def\eea{\end{eqnarray}}
\def\beaa{\begin{eqnarray*}}  \def\eeaa{\end{eqnarray*}}

\numberwithin{equation}{section}

\newtheorem{remark}[theorem]{Remark}
\numberwithin{equation}{section}

\begin{document}

\title[Chemotaxis effect, logistic source, and boundedness]
{Chemotaxis effect vs logistic damping  on boundedness in the 2-D minimal Keller-Segel model}

\author{Hai-Yang Jin}
\address{School of Mathematics, South China University of Technology, Guangzhou 510640, China}
\email{mahyjin@scut.edu.cn}
\author{Tian Xiang}
\address{Institute for Mathematical Sciences, Renmin University of China, Bejing, 100872, China}
\email{txiang@ruc.edu.cn}




\begin{abstract}

We study  chemotaxis effect  vs logisticdamping on  boundedness for the two-dimensional  minimal Keller-Segel model with logistic source:
$$\left\{ \begin{array}{lll}
&u_t =  \nabla \cdot (\nabla  u-\chi u\nabla v)+ u-\mu u^2,  &\quad x\in \Omega, t>0, \\[0.2cm]
& v_t =\Delta v -v+u,  &\quad x\in \Omega, t>0 \end{array}\right.
$$
 in a  smooth bounded  domain  $\Omega\subset \mathbb{R}^2$  with  $\chi, \mu>0$, nonnegative initial data $u_0, v_0$ and homogeneous Neumann boundary data. It is well-known that this model allows only for global and uniform-in-time bounded solutions for any $\chi, \mu>0$. Here,  we  carefully employ a simple and new method to regain its  boundedness and,  with particular attention to how upper bounds of solutions  qualitatively depend on $\chi$ and $\mu$. More, precisely, it is shown  there exists $C=C(u_0, v_0, \Omega)>0$ such that
$$
\|u(\cdot,t)\|_{L^\infty(\Omega)}\leq C\Bigr[1+\frac{1}{\mu}+\chi K(\chi,\mu) N (\chi,\mu) \Bigr]
$$
and
$$
\|v(\cdot, t)\|_{W^{1,\infty}(\Omega)}\leq C\Bigr[1+\frac{1}{\mu}+\frac{\chi^\frac{8}{3}}{\mu}K^\frac{8}{3}(\chi,\mu)\Bigr]=:CN(\chi,\mu)
$$
 uniformly  on  $[0, \infty)$,  where
$$
K(\chi, \mu)=M(\chi,\mu)E(\chi,\mu), \quad M(\chi,\mu)=1+\frac{1}{\mu}+\sqrt{\chi}(1+\frac{1}{\mu^2})
$$
and
$$
E(\chi,\mu)=\exp\Bigr[\frac{\chi C_{GN}^2}{2\min\{1,\frac{2}{\chi}\}}\Bigr(\frac{4}{\mu}\|u_0\|_{L^1(\Omega)}+\frac{13}{2\mu^2}|\Omega|
+\|\nabla v_0\|^2_{{L^2}(\Omega)}\Bigr)\Bigr].
$$
We notice that these upper bounds  are increasing in $\chi$, decreasing in $\mu$ and have only one singularity at $\mu=0$, where  the corresponding minimal model (removing the term $+u-\mu u^2$ in the first equation) is widely  known to possess blow-ups for large initial data.
\end{abstract}

\maketitle

\section{Introduction and main results}
In this work, we are concerned with the well-known and extensively explored  Keller-Segel minimal chemotaxis model with logistic source:
 \be \label{ks-asc-pp}\begin{cases}
u_t =   \nabla \cdot (\nabla  u-\chi u\nabla v)+ ru-\mu u^2  &x\in \Omega, t>0, \\[0.2cm]
 v_t =  \Delta v- v+ u,  & x\in \Omega, t>0, \\[0.2cm]
\frac{\partial u}{\partial \nu}=\frac{\partial v}{\partial \nu}=0, & x\in \partial \Omega, t>0,\\[0.2cm]
u(x,0)=u_0(x)\geq 0,  v(x,0)= v_0(x)\geq 0, & x\in \Omega,    \end{cases}  \ee
where $\Omega \subset \mathbb{R}^n(n\geq 1)$ is a bounded smooth domain,  $r\geq 0,  \chi, \mu>0$ and  $u$ and $v$ respectively denotes the density of cells and the concentration of the chemical signal.  The  chemotactic flux   $-\chi  u \nabla v$ (defining term in chemotaxis models) models the directed movement that $u$ moves towards the higher concentration of  $v$. This is commonly termed as chemotactic movement, a biological phenomenon whereby biological individuals orient their movement in response to some external signaling substances which attract cells to aggregate.

  Without logistic source, i.e., $r=0, \mu=0$,   the system \eqref{ks-asc-pp} is known as  the classical Keller-Segel minimal model \cite{Ke}, which and whose variants have been widely investigated since 1970.  The striking  feature of KS type models is the possibility of blow-up of solutions in a finite/infinite time, which strongly depends on the space dimension. A finite/infinite  time blow-up never occurs in $1$-D  \cite{HP04, OY01, Xiang14},  a  critical  mass blow-up occurs in $2$-D: when the initial mass $\|u_0\|_{L^1}<4\pi/\chi$,  solutions exist globally and converge to a single equilibrium in large time, whereas,  when $\|u_0\|_{L^1}>4\pi/\chi$, there exist solutions  blowing up in finite time, cf. \cite{HW01,FLP07, NSY97, Na01, SS01}, and even small initial mass can result in  blow-ups in $\geq 3$-D \cite{Win10-JDE, Win13}. See  \cite{BBTW15, Ho1}  for more surveys on the classical KS model and its variants.

The logistic source was introduced  by Mimura and Tsujikawa \cite{MT96}, where they study aggregating patterns based on  the chemotaxis, diffusion and growth of bacteria. First, this additional logistic term apparently destroys the conservation law of mass of the  classical KS model. On the other hand, it exerts a certain growth-inhibiting influence on  the global existence and boundedness of solutions to the corresponding Keller-Segel models. Indeed, in the case $n=1,2$,  even arbitrarily  small $\mu>0$ will be enough to prevent blow-ups by guaranteeing all solutions to  \eqref{ks-asc-pp} are global-in-time and uniformly bounded \cite{ HP04, OY01, OTYM02, Xiangjde}. This is even  true for a $2$-D simplified version of parabolic-elliptic (the second PDE in \eqref{ks-asc-pp} is replaced with $ 0=\Delta v-v+u$) chemotaxis system with singular sensitivity \cite{FW14}.  Whereas, in the case $n\geq 3$,  the global existence and  boundedness  were first  obtained  for a parabolic-elliptic simplification of \eqref{ks-asc-pp} under $\mu>\frac{(n-2)}{n}\chi$ \cite{TW07}. Nowadays,  this result has been improved to the borderline case $\mu\geq \frac{n-2}{n}\chi$ \cite{HT16, KS16, WXpre}. Moreover, with a very slow self-diffusion of cells,  the $u$ component can exceed the carrying capacity $\frac{r}{\mu}$ to an arbitrary extent at some intermediate time \cite{La15-DCDS,Win14-JNS}. Coming back to our  parabolic mode \eqref{ks-asc-pp},  for $\Omega$ being convex, Winkler first derived  the boundeness and global existence provided that $\mu$ is beyond  a certain number $\mu_0$ not explicitly known \cite{Win10}. A further  progress in this regard  was derived as long as $\mu>\theta_0 \chi$ for some implicit positive constant  $\theta_0$ in  \cite{YCJZ15}. An explicit lower bound for a $3$-D chemotaxis-fluid system with logistic source, when applied to \eqref{ks-asc-pp} with $\chi=1$, which states that $\mu\geq 23$ is enough to ensure boundedness \cite{TWZAMP}. This  bound $\mu_0$ was further improved by  Lin and Mu  \cite{ML16}  in $3$-D, wherein they replaced the logistic source in \eqref{ks-asc-pp} by thedamping term $u-\mu u^r$ with $r\geq 2$ to derive the boundedness  under $\mu^{\frac{1}{r-1}} >20\chi $. Very recently, for a full-parameter version of \eqref{ks-asc-pp}, we calculate out the explicit formula for $\mu_0$ in terms of the involving parameters, which states that $\mu>\frac{9}{\sqrt{10}-2}\chi=(7.743416\cdots)\chi$ ensures  boundedness and global existence for \eqref{ks-asc-pp} in $3$-D \cite{Xiangpre2}. Yet, it is a big open challenging problem whether or not blow-up occurs in \eqref{ks-asc-pp}  for small $\mu>0$, even though the existence of global weak solutions is available in convex 3-D domains  for $\mu>0$ \cite{La15-JDE}. Under further conditions on $\chi, \mu$ or $r$, convergence of bounded solutions to  the constant equilibrium $(\frac{r}{\mu}, \frac{r}{\mu})$ as well as its convergence rates are available \cite{HZ16, ML16, Win14,  Xiangpre2}. It also needs to be mentioned that for certain choices of the parameters, the solutions of \eqref{ks-asc-pp} even may oscillate drastically in
time, as numerically illustrated in \cite{HP11}, and that the solutions may undergo  transient growth phenomena, as demonstrated in \cite{La15-DCDS,Win14-JNS, Win17-DCDSB}.

 In contrast to the rich knowledge on boundedness, convergence and other dynamical properties for \eqref{ks-asc-pp} and its variants,   understanding the qualitative or quantitative properties
even of bounded solutions to chemotaxis problems seems much less developed. In this direction, a work was considered by Tao and Winkler in \cite{TW15-JDE} to show the mass persistence phenomenon for \eqref{ks-asc-pp},   i.e, for any supposedly given global classical and bounded nontrivial solution $(u,v)$ of \eqref{ks-asc-pp},  there is $m_*>0$ such that $\|u(t)\|_{L^1}\geq m_*$ for all $t>0$.  To our best knowledge, there seems no work on how boundedness or upper bounds of solutions of \eqref{ks-asc-pp} depends on the system parameters, say, $\chi, \mu$ or $r$. In this paper, we aim as  a first step  to  study  chemotaxis effect  vs logistic damping on  boundedness for the minimal chemotaxis-logistic  model  \eqref{ks-asc-pp} in 2-D. We do so partially because all solutions in 2-D are global and bounded by \cite{OTYM02, Xiangjde}. We are particularly interested in  the dependence of  upper bounds of solutions to \eqref{ks-asc-pp} on the most interesting  parameters $\chi$ and $\mu$. We hope that this qualitative boundedness would stimulate new research directions, especially, the same problem in higher dimensions. Since the constant $r$ doesn't bother us much in our derivation, we include it here. With this goal in mind, our main qualitative boundedness result reads as follows:

\begin{theorem}\label{bdd-2d} Let   $\chi, \mu>0, r\geq 0$,   $\Omega\subset \mathbb{R}^2$ be a bounded domain with a smooth boundary and and let the initial data $u_0\in C(\overline{\Omega})$  and $v_0\in W^{1,\infty}(\Omega)$ be nonnegative. Then the  Keller-Segel  chemotaxis-logistic model  \eqref{ks-asc-pp}  has a  unique global classical nonnegative  solution $(u,v)$ on $\Omega\times [0, \infty)$ for which
\be\label{u-infty2}
\begin{split}
\|u(t)\|_{L^\infty(\Omega)}\leq C\Bigr[1+\frac{1}{\mu}+\chi K(\chi,\mu) N (\chi,\mu) \Bigr]=:CL(\chi,\mu)
 \end{split}
\ee
and
\be\label{v-w1infty2}
\|v(t)\|_{W^{1,\infty}(\Omega)}\leq C\Bigr[1+\frac{1}{\mu}+\frac{\chi^\frac{8}{3}}{\mu}K^\frac{8}{3}(\chi,\mu)\Bigr]=:CN(\chi,\mu)
\ee
 uniformly  on  $[0, \infty)$ and for some $C$ depending on $u_0,v_0, r$ and $|\Omega|$, where
\be\label{M-def}
K(\chi, \mu)=M(\chi,\mu)E(\chi,\mu), \quad M(\chi,\mu)=1+\frac{1}{\mu}+\sqrt{\chi}(1+\frac{1}{\mu^2})
\ee
and
\be\label{E-def}
E(\chi,\mu)=e^{\frac{\chi C_{GN}^2}{2\min\{1,\frac{2}{\chi}\}}\Bigr[\frac{(r+3)}{\mu}\|u_0\|_{L^1(\Omega)}+\frac{(r+1)^3}{4\mu^2}|\Omega|
+\|\nabla v_0\|_{{L^2}(\Omega)}^2+ \frac{(r+2)^2}{2\mu^2} |\Omega|\Bigr]}.
\ee
\end{theorem}
Up to a scaling constant,  Theorem \ref{bdd-2d} provides explicit upper bounds  for $\|u(t)\|_{L^\infty}$ and $\|v(t)\|_{W^{1,\infty}}$ in terms of  the most interesting parameters $\chi$ and $\mu$ in $2$-D.

The crucial point of the proof of Theorem \ref{bdd-2d}  consists in deriving a uniform-in-time estimate for $\|u(t)\|_{L^2}$ rather than $\|(u+1)\ln (u+1)\|_{L^1}$ as in \cite{OTYM02,  Xiangjde}; indeed, we obtain  an  explicit  uniform-in-time bound  for $\|u(t)\|_{L^2}^2$ as follows:
\be\label{ul2-est0}
\begin{split}
\|u(t)\|_{L^2}^2&\leq
E(\chi,\mu)\Bigr\{\|u_0\|_{L^2}^2+\frac{8\min\{1,\frac{2}{\chi}\}}{ C_{GN}^2}+\frac{(r+1)}{\mu}\|u_0\|_{L^1}\\
&\quad \quad +\frac{3\chi C_{GN}^2}{4}\Bigr[\|u_0\|_{L^1}+\frac{(r+1)^2}{4\mu}|\Omega|\Bigr]^4
+\frac{(r+1)^3}{4\mu^2}|\Omega|+\frac{8r^3}{9\mu^2}|\Omega|\Bigr\},
\end{split}
\ee
where $E$ is defined by \eqref{E-def} and  $C_{GN}$ is the Gagliardo-Nirenberg  constant.  After obtaining precise  bounds on
$\|u\|_{L^1}, \|\nabla v\|_{L^2}$ and space-time integrals on $u^2$ and $|\Delta v|^2$, cf. Lemmas \ref{ul1-vgradl2} and \ref{ul2-v del2-st}, we can use the 2-D  Gagliardo-Nirenberg interpolation inequality to derive a
differential inequality  for $y(t)=\| u(t)\|_{L^2}^2+a$ of the form:
$$
y^\prime (t)\leq ky(t)z(t)+b, \quad z(t)=\|\Delta v(t)\|_{L^2}^2
$$
for some $a, k, b>0$ and then solving this ODI successively and using the gained space-time bounds, we achieve the desired estimate \eqref{ul2-est0}. This is inspired by the ideas presented in \cite[Lemma 3.4]{STW14}. Thanks to the $L^{\frac{n}{2}+}$-boundedness criterion in \cite{BBTW15, Xiangjde},  the uniform-in-time bound for $\|u(t)\|_{L^2}$ indeed implies the global existence and  boundedness. While, to dig out the dependence of boundedness on $\chi$ and $\mu$, we first use the established $L^2$-estimate of $u$ together with a widely used 'reciprocal' lemma obtained from  the $v$-equation, cf, Lemma \ref{reciprocal-lem} to bound $\|\nabla v\|_{L^q}$ for any $q\in (1,\infty)$, and then,  we test the $u$-equation in \eqref{ks-asc-pp} by $u^2$ to derive the $L^3$-estimate of $u$, and finally, we apply the variation-of-constants formula for  $u$ and $v$ and use the well-known smoothing $L^p$-$L^q$ type estimates for the Neumann heat semigroup in $\Omega$, cf. \cite{Cao15, Win10-JDE} to conclude the respective bounds for $(u,v)$  in  \eqref{u-infty2} and \eqref{v-w1infty2}.

\begin{remark}From   \eqref{u-infty2}, \eqref{v-w1infty2}, \eqref{M-def} and \eqref{E-def},  one can see that $L, M, N, K, E$ defined on $[0,\infty)\times (0,\infty)$ are decreasing in $\mu$ and are increasing in $\chi$    have only one singularity at $\mu=0$.  Therefore,  our obtained bounds for  $\|u(t)\|_{L^\infty}$ and $\|v(t)\|_{W^{1,\infty}}$ enjoy these properties. It is worthwhile to observe that,   when $\mu=0$, even $r=0$, the corresponding KS minimal model possesses  blow-ups for large initial data \cite{HW01, NSY97, Na01, SS01}, illustrating the reasonableness of adding a logistic source to the KS minimal model to prevent blow-up.
\end{remark}

\section{Preliminaries}

For convenience, we start with  the well-known Young's inequality with $\epsilon$:
\begin{lemma}\label{Y-epsilon}(Young' s inequality with $\epsilon$) Let $p$  and $q$ be two given positive numbers with $\frac{1}{p}+\frac{1}{q}=1$. Then, for any  $\epsilon>0$, it holds
$$
ab\leq \epsilon a^p+\frac{b^q}{(\epsilon p)^{\frac{q}{p}}q},  \quad \quad \forall a,b\geq 0.
$$\end{lemma}

\begin{lemma}\label{GN-inter}(Gagliardo-Nirenberg interpolation inequality \cite{Fried,  Nirenberg66}) Let $p\geq 1$ and  $q\in (0,p)$. Then there exist a positive constant   $C_{GN}$ depending on $p$ and $q$ such that
 $$
 \|w\|_{L^p} \leq C_{GN}\Bigr(\|\bigtriangledown w\|_{L^2}^{\delta}\|w\|_{L^q}^{(1-\delta) }+\|w\|_{L^s}\Bigr), \quad \forall w\in H^1\cap L^q,
 $$
where  $s>0$ is arbitrary and  $\delta$ is given by
 $$
 \frac{1}{p}=\delta(\frac{1}{2}-\frac{1}{n})+\frac{1-\delta}{q}\Longleftrightarrow \delta=\frac{\frac{n}{q}-\frac{n}{p}}{1-\frac{n}{2}+\frac{n}{q}}\in(0,1).
 $$
\end{lemma}
 The basic result on  local existence, uniqueness   and extendibility   of classical solutions for the  minimal KS  system \eqref{ks-asc-pp} can be found in \cite[Lemma 1.1]{Win10}.
\begin{lemma}Let  $\chi, \mu>0, r\geq 0$,   $\Omega\subset \mathbb{R}^n (n\geq 1)$ be a bounded smooth domain and let the initial data $u_0\in C(\overline{\Omega})$  and $v_0\in W^{1,\infty}(\Omega)$ be  nonnegative. Then there is a unique,  nonnegative and  classical maximal solution $(u,v)$ of the IBVP \eqref{ks-asc-pp}  on some maximal interval $[0, T_m)$ with $0<T_m \leq \infty$ such that
$$\ba{ll}
u\in C(\overline{\Omega}\times [0, T_m))\cap C^{2,1}(\overline{\Omega}\times (0, T_m)), \\[0.2cm]
v\in C(\overline{\Omega}\times [0, T_m))\cap C^{2,1}(\overline{\Omega}\times (0, T_m))\cap L_{\text{loc}}^\infty([0, T_m); W^{1,s}(\Omega))
\ea $$
for any $s>n$. In particular,  if $T_m<\infty$, then
$$ \|u( t)\|_{L^\infty}+\|v( t)\|_{W^{1,s}}\rightarrow \infty \quad \quad \mbox{ as }  t\rightarrow T_m^{-}.
$$
\begin{lemma}\label{ul1-vgradl2}  For any $t\in [0, T_m)$, the nonnegative  solution $(u,v)$ of \eqref{ks-asc-pp} satisfies
\be\label{u-bdd}
\|u\|_{L^1} \leq \|u_0\|_{L^1}+\frac{(r+1)^2}{4\mu}|\Omega|=:k_1\ee
and
\be\label{u-gradv-bdd}
 \|\nabla v\|_{L^2}^2\leq \frac{2}{\mu} \Bigr[\|u_0\|_{L^1}+\frac{\mu}{2}\|\nabla v_0\|_{L^2}^2+ \frac{(r+2)^2}{4\mu} |\Omega|\Bigr]=:k_2.
\ee
\end{lemma}
\begin{proof}The nonnegativity of $u,v$  follows from the maximum principle. Then integrating the $u$-equation and using the homogeneous Neumann boundary condition,  we derive
\be\label{ul1-diff}
\frac{d}{dt}\int_\Omega u= r \int_\Omega u-\mu \int_\Omega u^2\leq -\int_\Omega u+\frac{(r+1)^2}{4\mu}|\Omega|,
\ee
which yields the $L^1$-bound for $u$ in \eqref{u-bdd}.

Then testing the $v$-equation in  \eqref{ks-asc-pp} against $-\Delta v$ and integrating by parts and using Youn's inequality with epsilon as stated in Lemma \ref{Y-epsilon}, we obtain
\be\label{v-test}
\frac{1}{2}\frac{d}{dt}\int_\Omega |\nabla v|^2+\frac{1}{2}\int_\Omega |\Delta v|^2\leq -\int_\Omega |\nabla v|^2+\frac{1}{2}\int_\Omega u^2,
\ee
which together with the reasoning leading to \eqref{ul1-diff}  gives us
$$
 \frac{d}{dt}\int_\Omega (u+\frac{\mu}{2} |\nabla v|^2) +2\int_\Omega (u+\frac{\mu}{2}|\nabla v|^2)\leq \frac{(r+2)^2}{2\mu}|\Omega|.
$$
Solving this standard Gronwall inequality shows
$$
\|u\|_{L^1}+\frac{\mu}{2}\|\nabla v\|_{L^2}^2\leq  \|u_0\|_{L^1}+ \frac{\mu}{2}\|\nabla v_0\|_{L^2}^2+ \frac{(r+2)^2}{4\mu} |\Omega|,
$$
which directly leads to \eqref{u-gradv-bdd}.
\end{proof}
\end{lemma}
\section{Chemotaxis vs logistic  on Boundedness in 2-D}

In 2-D, it is well-known that any presence of logistic source  will be sufficient to suppress blow-up by ensuring all solutions to  \eqref{ks-asc-pp} are global-in-time and uniformly bounded \cite{OTYM02, Xiangjde}.   In this section, we carefully scrutinize a different  method motivated from \cite[Lemma 3.4]{STW14} to regain its boundness and,  with particular focus on  the qualitative dependence of upper bounds of  solutions to  \eqref{ks-asc-pp}     on $\chi$ and $\mu$, and thus accomplish the proof of Theorem \ref{bdd-2d}.

\begin{lemma}\label{ul2-v del2-st}  Given $\tau \in(0, T_m)$, then, for any $t\in [0, T_m-\tau)$, the solution $(u,v)$ of the KS model \eqref{ks-asc-pp} fulfills
\begin{equation}\label{ul2-st}
\int_t^{t+\tau }\int_\Omega u^2\leq \frac{(r+1)k_1}{\mu}\max\{\tau, 1\}=:k_3\max\{\tau, 1\},
\end{equation}
\be\label{gradv-l2-st}
 \int_t^{t+\tau}\int_\Omega |\nabla v|^2\leq k_2\max\{\tau, 1\}
\ee
and
\be\label{del v-st}
 \int_t^{t+\tau}\int_\Omega |\Delta v|^2 \leq (k_3+ k_2) \max\{\tau, 1\}=:k_4\max\{\tau, 1\}.
\ee
\end{lemma}
\begin{proof}For any $t\in[ 0, T_m-\tau)$, integrating the $u$-equation in \eqref{ks-asc-pp} over $\Omega\times(t,t+\tau)$ and using Lemma \ref{ul1-vgradl2}, we deduce that
$$
\mu \int_t^{t+\tau}\int_\Omega u^2\leq r\int_t^{t+\tau}\int_\Omega u+\int_\Omega u\leq (r+1)k_1\max\{\tau, 1\},
$$
yielding the desired inequality \eqref{ul2-st}.

The estimate \eqref{gradv-l2-st} follows directly from \eqref{u-gradv-bdd}. Next, an integration of \eqref{v-test}  over $(t, t+\tau)$ and the use of   \eqref{u-gradv-bdd}  and \eqref{ul2-st} telescope
\begin{align*}
 \int_t^{t+\tau}\int_\Omega |\Delta v|^2(s)\leq \int_t^{t+\tau} \int_\Omega u^2(s)+\int_\Omega |\nabla v|^2(t) \leq (k_3+ k_2) \max\{\tau, 1\},
\end{align*}
which is exactly  \eqref{del v-st}.
\end{proof}
 Here, with Lemma \ref{ul2-v del2-st} at hand, in 2-D setting, we can make use of the Gagliardo-Nirenberg interpolation inequality in Lemma \ref{GN-inter} to derive an ODI satisfied by $\|u\|_{L^2}^2$, which enables us to deduce an estimate  for  $\|u\|_{L^2}$. This  is the key  point for us to derive qualitative bounds  for $\|u\|_{L^\infty}$ and $\|v\|_{W^{1,\infty}}$ later on.
\begin{lemma}\label{ul2-est}  Given $\tau \in(0, T_m)$, then the $u$-component of the solution $(u,v)$ of the KS minimal  model  \eqref{ks-asc-pp} satisfies the explicit  uniform-in-time bound:
\be\label{ul2-est}
\begin{split}
\|u(t)\|_{L^2}^2&\leq
\Bigr\{\|u_0\|_{L^2}^2+\frac{8\min\{1,\frac{2}{\chi}\}}{ C_{GN}^2}+\frac{3\chi C_{GN}^2}{4}\Bigr[\|u_0\|_{L^1}+\frac{(r+1)^2}{4\mu}|\Omega|\Bigr]^4\\
&+\frac{(r+1)}{\mu}\|u_0\|_{L^1}+\frac{(r+1)^3}{4\mu^2}|\Omega|+\frac{8r^3}{9\mu^2}|\Omega|\Bigr\}\max\{1, \tau, \frac{1}{\tau}\}\\
&\times e^{\frac{\chi C_{GN}^2}{2\min\{1,\frac{2}{\chi}\}}\Bigr[\frac{(r+3)}{\mu}\|u_0\|_{L^1}+\frac{(r+1)^3}{4\mu^2}|\Omega|
+\|\nabla v_0\|_{L^2}^2+ \frac{(r+2)^2}{2\mu^2} |\Omega|\Bigr]\max\{1,\tau\}},
\end{split}
\ee
and so a uniform estimate for $\|u\|_{L^2}$ in terms of $\chi$ and $\mu$ follows:
\be\label{ul2-est-rough}
 \|u(t)\|_{L^2}\leq C\Bigr[1+\frac{1}{\mu}+\sqrt{\chi}(1+\frac{1}{\mu^2})\Bigr]\max\{\sqrt{\tau}, \frac{1}{\sqrt{\tau}}\}E^{\max\{1,\tau\}}(\chi,\mu)
\ee
for all $t\in (0, T_m)$ and for some $C=C(u_0, r,|\Omega|)$, where $E$ is defined by \eqref{E-def}.
\end{lemma}
\begin{proof} We test the  $u$- equation in \eqref{ks-asc-pp} by $u$  and integrate  by parts to deduce from H\"{o}lder's  inequality  that
\be \label{u-lp-int}
\begin{split}
\frac{1}{2}\frac{d}{dt} \int_\Omega u^2+ \int_\Omega |\nabla u |^2& =\frac{\chi}{2} \int_\Omega\nabla (u^2) \nabla v+\int_\Omega u^2(r-\mu u)\\
&\quad =-  \frac{\chi}{2}\int_\Omega  u^2 \Delta  v +\int_\Omega u^2(r-\mu u) \\
&\quad \leq \frac{ \chi}{2} \Bigr(\int_\Omega  u^{4}\Bigr)^{\frac{1}{2}}  \Bigr(\int_\Omega  |\Delta  v|^2 \Bigr)^{\frac{1}{2}} +\int_\Omega u^2(r-\mu u).   \end{split}
\ee
 Applying the  GN interpolation inequality  in Lemma \ref{GN-inter} with $n=2$ and  the boundedness of $\|u\|_{L^1}$ in  \eqref{u-bdd}, we estimate
\begin{align*}
\Bigr(\int_\Omega  u^{4}\Bigr)^{\frac{1}{2}}=\|u \|_{L^4}^2&\leq C_{GN}\Bigr(\|\nabla u\|_{L^2}\|u\|_{L^2}+\|u\|_{L^1}^2\Bigr)\\
& \leq C_{GN}\Bigr(\|\nabla u\|_{L^2}\|u \|_{L^2}+k_1^2\Bigr).
\end{align*}
Hence, upon twice uses   of Young's inequality with epsilon, cf. Lemma \ref{Y-epsilon}, for any $\epsilon>0$,  it follows that
\be\label{GN-Young-bdd-2d}
\begin{split}
&\Bigr(\int_\Omega  u^{4}\Bigr)^{\frac{1}{2}}  \Bigr(\int_\Omega  |\Delta  v|^2 \Bigr)^{\frac{1}{2}} \\
&\leq  C_{GN} \|\nabla u\|_{L^2}\|u\|_{L^2} \|\Delta v\|_{L^2}+k_1^2C_{GN}\|\Delta v\|_{L^2}\\
&\leq \epsilon \|\nabla u\|_{L^2}^2+\frac{C_{GN}^2}{4\epsilon} \|u\|_{L^2}^2\|\Delta v\|_{L^2}^2 +\|\Delta v\|_{L^2}^2+\frac{k_1^{4}C_{GN}^2}{4}.
\end{split}
\ee
Inserting \eqref{GN-Young-bdd-2d} into \eqref{u-lp-int},  we conclude that
\begin{align*}
&\frac{1}{2}\frac{d}{dt} \int_\Omega u^2+ \int_\Omega |\nabla u|^2\\
&\leq  \frac{\chi}{2}\Bigr[\epsilon \|\nabla u\|_{L^2}^2+\frac{C_{GN}^2}{4\epsilon} \|u\|_{L^2}^2\|\Delta v\|_{L^2}^2+\|\Delta v\|_{L^2}^2+\frac{k_1^{4}C_{GN}^2}{4}\Bigr]\\
&\quad +\int_\Omega u^2(r-\mu u),    \end{align*}
from which, upon setting
\be\label{episilon-exp}
\epsilon=\min\{1,\frac{2}{\chi}\},
\ee
we deduce
\be\label{ul2-2d}
\begin{split}
\frac{d}{dt} \int_\Omega u^2&\leq \frac{\chi C_{GN}^2}{4\epsilon} \Bigr(\|u\|_{L^2}^2+\frac{4\epsilon}{C_{GN}^2} \Bigr) \|\Delta v\|_{L^2}^2+\frac{\chi k_1^{4}C_{GN}^2}{4}+\frac{8r^3}{27\mu^2}|\Omega|\\
&=: k_5y(t)z(t)+k_6,  \end{split}
\ee
 where  $k_5=\frac{\chi C_{GN}^2}{4\epsilon}$, $k_6=\frac{\chi k_1^{4}C_{GN}^2}{4}+\frac{8r^3}{27\mu^2}|\Omega|$ and
\begin{equation}\label{yz}
y(t)=\|u\|_{L^2}^2+\frac{4\epsilon}{C_{GN}^2} ,\ \ \  z(t)= \|\Delta v\|_{L^2}^2.
\end{equation}
For any $s\geq 0$ and any $t\geq s$, multiplying the integrating factor $\exp(-k_5\int_s^tz(\lambda)d\lambda)$ on both sides of  \eqref{ul2-2d},  we deduce that
\be\label{GW-2d}
y(t)\leq y(s)e^{k_5\int_s^tz(\sigma)d\sigma}+
k_6\int_s^te^{k_5\int_\xi^tz(\sigma)d\sigma}d\xi, \quad \forall t\in [s,\infty)\cap [0, T_m).
\ee
In view of \eqref{ul2-st} and \eqref{del v-st} in Lemma \ref{ul2-v del2-st} and the mean value theorem, one infers from the definitions of $y$ and $z$ in \eqref{yz} that
\be\label{MVT}
y(s_i)=\frac{1}{\tau}\int_{i\tau}^{(i+1)\tau} y(s)ds\leq (k_3+\frac{4\epsilon}{ C_{GN}^2})\max\{1, \frac{1}{\tau}\}=:k_7\max\{1, \frac{1}{\tau}\}
\ee
and
\be\label{MVT2}
 \int_{i\tau}^{(i+1)\tau} z(s)ds\leq k_4\max\{\tau, 1\}
\ee
for some $s_i\in [i\tau, (i+1)\tau]$ and any nonnegative  integers $i<\frac{T_m}{\tau}-1$.

First, for $t\in [0, \tau]$, we set $s=0$ in \eqref{GW-2d} and $i=0$ in \eqref{MVT2} to infer
\be \label{y-tau}
\begin{split}
y(t)&\leq y(0)e^{k_5\int_0^\tau z(\sigma)d\sigma}+
k_6\int_0^\tau e^{k_5\int_0^\tau z(\sigma)d\sigma}d\xi\\
&\leq  (y(0)+k_6)\max\{\tau, 1\}e^{k_5 k_4\max\{\tau, 1\}}.
\end{split}
\ee
Next, for $t\in [\tau, 2\tau]$, we will always assume that $t<T_m$, we put $s=s_0\in [0,\tau]$ in  \eqref{GW-2d}  to deduce from \eqref{MVT} and \eqref{MVT2}  that
\be \label{y-2tau}
\begin{split}
y(t)&\leq y(s_0)e^{k_5\int_{s_0}^t z(\sigma)d\sigma}+
k_6\int_{s_0}^t e^{k_5\int_{s_0}^t z(\sigma)d\sigma}d\xi\\
&\leq y(s_0)e^{k_5\int_{0}^{2\tau}z(\sigma)d\sigma}+
k_6\int_{0}^{2\tau} e^{k_5\int_{0}^{2\tau} z(\sigma)d\sigma}d\xi\\
&\leq  (k_7+2k_6)\max\{1, \tau, \frac{1}{\tau}\}e^{2k_5 k_4\max\{\tau, 1\}}.
\end{split}
\ee
In general,   for any $t\in(\tau, T_m)$, one first chooses $i\geq 0$ such that $t\in[(i+1)\tau, (i+2)\tau]$ and set $s=s_i\in [i\tau, (i+1)\tau]$  in \eqref{GW-2d}, and then infers from \eqref{MVT} and \eqref{MVT2}  that
\be \label{y-ntau}
\begin{split}
y(t)&\leq y(s_i)e^{k_5\int_{s_i}^t z(\sigma)d\sigma}+
k_6\int_{s_i}^t e^{k_5\int_{s_i}^t z(\sigma)d\sigma}d\xi\\
&\leq y(s_i)e^{k_5\int_{i\tau }^{(i+2)\tau}z(\sigma)d\sigma}+
k_6\int_{i\tau}^{(i+2)\tau} e^{k_5\int_{i\tau}^{(i+2)\tau} z(\sigma)d\sigma}d\xi\\
&\leq  (k_7+2k_6)\max\{1, \tau, \frac{1}{\tau}\}e^{2k_5 k_4\max\{\tau, 1\}}.
\end{split}
\ee
 Recalling from  the definition of $y(t)$ in \eqref{yz}, we then conclude  from \eqref{y-tau}, \eqref{y-2tau} and \eqref{y-ntau}  the uniform $L^2$-estimate of $u$:
\begin{align*}
\|u\|_{L^2}^2+\frac{4\epsilon}{C_{GN}^2}&\leq (y(0)+k_7+3k_6)\max\{1, \tau, \frac{1}{\tau}\}e^{2k_5 k_4\max\{\tau, 1\}}\\
&=
\Bigr(\|u_0\|_{L^2}^2+\frac{8\min\{1,\frac{2}{\chi}\}}{ C_{GN}^2}+k_3+\frac{3\chi k_1^4 C_{GN}^2}{4}
+\frac{8r^3}{9\mu^2}|\Omega|\Bigr)\\
&\quad \times  \max\{1, \tau, \frac{1}{\tau}\}e^{\frac{\chi C_{GN}^2}{2\min\{1,\frac{2}{\chi}\}}(k_3+ k_2)\max\{1,\tau\}},
\end{align*}
where we have substituted the definitions of  $k_4,k_5, k_6, k_7$ and $\epsilon$ in \eqref{del v-st}, \eqref{ul2-2d}, \eqref{MVT} and \eqref{episilon-exp}. In the above inequality, a further substitution of $k_1,k_2,k_3$ as defined in \eqref{u-bdd}, \eqref{u-gradv-bdd} and \eqref{ul2-st} yields the desired $L^2$-estimate of $u$ in  \eqref{ul2-est}.
\end{proof}
\begin{remark}\label{global-time} Another way to view the uniform $L^2$-norm of $u$ could be arguing as follows: Assume $T_m<\infty$. Then, for any given  large  natural number $N\gg 1$,  we set  $\tau=\frac{T_m}{N}$ so that $N\tau=T_m$. Then as arguing above we can obtain  that $\|u(t)\|_{L^2}$ is uniformly bounded in $(0, T_m)$, which violates the $L^{\frac{n}{2}+}$-criterion in \cite{BBTW15, Xiangjde} with $n=2$. Hence, $T_m=\infty$ and  $\|u(t)\|_{L^\infty}$ is uniformly bounded on $(0, \infty)$.  Furthermore,  this energy  method offers a simple proof for global-in-time boundedness in 2-D setting compared to existing literature, c.f. \cite{OTYM02, Xiangjde}.
\end{remark}
In the sequel,  we shall seek how the $(L^\infty, W^{1,\infty})$-bound of $(u,v)$  depends on $\chi$ and $\mu$. Since the solution $(u,v)$ is global in time by Remark \ref{global-time}, we will set $\tau=1$ to simplify our calculations. To get higher order regularity of $u$, we control the $W^{1, q}$-bounds of $v$ in terms of $L^p$-norms of $u$. For this purpose, we shall utilize the widely known smoothing $L^p$-$L^q$ properties of the Neumann heat semigroup $\{e^{t\Delta}\}_{t\geq0}$ in $\Omega$, see, e.g. \cite{Win10-JDE, Cao15} for instance.  Applying these heat Neumann semigroup estimates to the $v$-equation in \eqref{ks-asc-pp}, we have the following widely known 'reciprocal' lemma, cf. \cite[Lemma 4.1]{HW05}, \cite[Lemma 1]{KS08}, \cite[Lemma 3.5]{Xiangjde} for instance.

\begin{lemma}  \label{reciprocal-lem} For $p\geq 1$, let
\be\label{q-exp}
\begin{cases}
q\in [1, \frac{np}{n-p}), \quad & \text{if }  p\leq n,\\
q\in [1, \infty], \quad & \text{if }  p>n.
\end{cases}
\ee
Then there exists $C=C(p,q, v_0, \Omega)$ such that  the solution $(u,v)$ of \eqref{ks-asc-pp} satisfies
 \be\label{gradv-bdd-lp}
\| v(t)\|_{W^{1,q}} \leq C(1+\sup_{s\in(0,t)}\|u(s)\|_{L^p}).
\ee
 \end{lemma}
 \begin{proof}
Indeed, the variation-of-constant formula applied to the $v$ in \eqref{ks-asc-pp} gives
\be\label{v-var}
v(t)=e^{t(\Delta-1)}v_0+\int_0^te^{(t-s)(\Delta-1)}u(s)ds.
\ee
Now, the well-known  $L^p$-$L^q$ estimate for the heat Neumann semigroup guarantees, cf. \cite{Cao15, Win10-JDE},   for  $1\leq q\leq p\leq\infty$,   one can find $k_8,k_9, k_{10}>0$ such that
\be\label{sem-pro1}
\|e^{t\Delta}w\|_{L^p}\leq k_8\left(1+t^{-\frac{n}{2}(\frac{1}{q}-\frac{1}{p})}\right)\|w\|_{L^q}, \quad  \forall t>0
\ee
and
\be\label{sem-max}
\|\nabla e^{t\Delta} w\|_{L^q}\leq k_9\| \nabla w\|_{L^\infty}, \quad \forall t>0
\ee
as well as
\be\label{sem-lp-lq}
\|\nabla e^{t\Delta} w\|_{L^p}\leq k_{10}\left(1+t^{-\frac{1}{2}-\frac{n}{2}(\frac{1}{q}-\frac{1}{p})}\right)e^{-\lambda_1 t}\|w\|_{L^q}, \quad \forall t>0.
\ee
Here, $\lambda_1(>0)$ is the first nonzero eigenvalue of $-\Delta$ under homogeneous boundary condition.
Then applying the properties  \eqref{sem-pro1}, \eqref{sem-max} and \eqref{sem-lp-lq} to  \eqref{v-var} and the exponent relation $p,q$ in  \eqref{q-exp}, one can easily infer \eqref{gradv-bdd-lp}.
\end{proof}
\begin{lemma}  The $u$-component of the solution  of the KS minimal chemotxis-logistic model \eqref{ks-asc-pp} satisfies the uniform estimate
\be\label{ul3-est-rough}
 \|u(t)\|_{L^3}\leq C\Bigr[1+\frac{1}{\mu}+\frac{\chi^\frac{8}{3}}{\mu}M^\frac{8}{3}(\chi,\mu)E^\frac{8}{3}(\chi,\mu)\Bigr],
\ee
for all $t\in (0, \infty)$ and for some $C$ depending on $u_0,v_0, r$ and $|\Omega|$, where $M$ and $E$ are defined by \eqref{M-def} and \eqref{E-def}, respectively.
\end{lemma}
\begin{proof}
Based on the uniform $L^2$-bound of $u$ in  \eqref{ul2-est-rough} with $\tau=1$,  it follows from Lemma \ref{reciprocal-lem} with $n=2$ and $p=2$, for any $1<q<\infty$, that
 \be\label{gradv-bdd-lq2}
\|\nabla v(t)\|_{L^q} \leq  CM(\chi,\mu)E(\chi,\mu).
\ee
Multiplying the $u$-equation in \eqref{ks-asc-pp} by $u^2$, integrating by parts and using Young's inequality with epsilon, we arrive at
\begin{align*}
&\frac{1}{3}\frac{d}{dt} \int_\Omega u^3+2\int_\Omega  u |\nabla u|^2\\
&=2\chi \int_\Omega u^2\nabla u\nabla v+\int_\Omega (ru^3-\mu u^4)\\
& \leq 2\int_\Omega  u |\nabla u|^2+\frac{ \chi^2}{2} \int_\Omega u^3|\nabla v|^2+ \int_\Omega (ru^3-\mu u^4)\\
&\leq 2\int_\Omega  u |\nabla u|^2+\frac{\mu}{2}\int_\Omega u^4+\frac{3^3 \chi^8}{2\cdot 4^4\mu^3} \int_\Omega |\nabla v|^8+ \int_\Omega (ru^3-\mu u^4), \end{align*}
 which along with the algebraic fact $ru^3-\frac{\mu}{2}u^4\leq -\frac{1}{3}u^3+\frac{3^3(r+\frac{1}{3})^4}{2^5\mu^3}$ shows that
$$\frac{d}{dt} \int_\Omega u^3+ \int_\Omega u^3 \leq \frac{3^4 \chi^8}{2\cdot 4^4\mu^3}\|\nabla v\|_{L^8}^8+\frac{3^4(r+\frac{1}{3})^4}{2^5\mu^3}|\Omega|.
$$
Solving this standard Gronwall differential inequality, we  directly have
$$
\|u\|_{L^3}^3\leq \|u_0\|_{L^3}^3+\frac{3^4 \chi^8}{2\cdot 4^4\mu^3}\sup_{t\in (0, \infty)}\|\nabla v(t)\|_{L^8}^8+\frac{3^4(r+\frac{1}{3})^4}{2^5\mu^3}|\Omega|,
$$
which together with \eqref{gradv-bdd-lq2} with $q=8$ yields the desired estimate \eqref{ul3-est-rough}.
\end{proof}

\begin{proof}[Proof of  Theorem \ref{bdd-2d}]  The $W^{1,\infty}$-bound of $v$ in \eqref{v-w1infty2} follows directly from the uniform $L^3$-estimate of $u$ in \eqref{ul3-est-rough} and Lemma \ref{reciprocal-lem} with $(n,p,q)=(2,3,\infty)$.

For the $L^\infty$-bound of $u$, we first apply the variation-of-constants formula to the $u$-equation in \eqref{ks-asc-pp} to represent  $u$  as
\be\label{u-vcf2}
\begin{split}
u(t)&=e^{t(\Delta-1) }u_0-\chi\int_0^te^{(t-s)(\Delta-1) }\nabla \cdot((u\nabla v)(s))ds \\
  &\quad +\int_0^te^{(t-s)(\Delta-1) }[(r+1)u(s)-\mu u^2(s)]ds\\
  &=:u_1(t)+u_2(t)+u_3(t).
  \end{split}
\ee
Because $u$ is nonnegative and smooth, we thus have
$$
\|u(t)\|_{L^\infty}=\sup_{x\in \Omega}u(x,t)\leq \sup_{x\in \Omega}u_1(x,t)+\sup_{x\in \Omega}u_2(x,t)+\sup_{x\in \Omega}u_3(x,t).
$$
 Thanks to the the maximum principle, the Neumann heat semigroup $(e^{t\Delta })_{t\geq 0}$  is order preserving. This allows us to control $u_1$ and  $u_3$ as follows:
 \be\label{u1-est}
\|u_1(t)\|_{L^\infty}=\|e^{t(\Delta-1) }u_0\|_{L^\infty}\leq e^{-t}\|u_0\|_{L^\infty}\leq \|u_0\|_{L^\infty}.
\ee
as well as
 \be\label{u3-est}
 \begin{split}
 u_3(t)&=\int_0^te^{-(t-s)}e^{(t-s)\Delta}[(r+1)u(s)-\mu u^2(s)]ds\\
 &\leq \int_0^te^{-(t-s) } e^{(t-s)\Delta }\frac{(r+1)^2}{4\mu}ds\leq \frac{(r+1)^2}{4\mu}.
 \end{split}
 \ee
To estimate $u_2$, we recall one more property of the Neumann heat semigroup  $e^{t\Delta}$, cf. \cite{Cao15,Win10-JDE}: for any  $1\leq q\leq p\leq \infty$, there exists $k_{11}>0$ such that
\be\label{sem-pro2}
\|e^{t\Delta}\nabla\cdot w\|_{L^p}\leq k_{11}\left(1+t^{-\frac{1}{2}-\frac{n}{2}(\frac{1}{q}-\frac{1}{p})}\right)e^{-\lambda_1 t}\|w\|_{L^q}, \forall t>0, w\in (W^{1,p})^n.
\ee
Using the definition of $u_2$ in \eqref{u-vcf2},  \eqref{sem-pro2} with $n=2$ and H\"{o}lder  interpolation inequality, we deduce that
\begin{align*}
\|u_2(t)\|_{L^\infty}&\leq  \chi\int_0^t \|e^{(t-s)(\Delta-1)} \nabla\cdot(u( s)\nabla v( s))\|_{L^\infty}ds \\
&\leq  k_{11}\chi\int_0^t (1+(t-s)^{-\frac{1}{2}-\frac{2}{5}})e^{-(\lambda_1+1) (t-s)}\|u(s)\nabla v(s)\|_{L^\frac{5}{2}}ds\\
&\leq k_{11}\chi\int_0^t (1+(t-s)^{-\frac{1}{2}-\frac{2}{5}})e^{-(\lambda_1+1) (t-s)} \|u(s)\|_{L^3}\|\nabla v(s)\|_{L^{15}}ds\\
&\leq k_{11}\chi \sup_{s\in(0,\infty)}[\|u(s)\|_{L^3}\|\nabla v(s)\|_{L^{15}}] \int_0^\infty  (1+\sigma^{-\frac{9}{10}})e^{-(\lambda_1+1) \sigma} d\sigma \\
&=:k_{12}\chi \sup_{s\in(0,\infty)}\|u(s)\|_{L^3}\sup_{s\in(0,\infty)}\|\nabla v(s)\|_{L^{15}}.
\end{align*}
This in conjunction with \eqref{ul3-est-rough} and  \eqref{gradv-bdd-lq2} with $q=15$ gives the estimate of $u_2$:
\be\label{u2-est}
\|u_2(t)\|_{L^\infty}\leq C\chi M(\chi,\mu)E(\chi,\mu) \Bigr[1+\frac{1}{\mu}+\frac{\chi^\frac{8}{3}}{\mu}M^\frac{8}{3}(\chi,\mu)E^\frac{8}{3}(\chi,\mu)\Bigr].
\ee
A substitution of \eqref{u1-est},  \eqref{u3-est} and \eqref{u2-est}  into \eqref{u-vcf2} yields the desired uniform bound for $\|u(t)\|_{L^\infty}$ as stated in \eqref{u-infty2}.
\end{proof}

\textbf{Acknowledgments}   The research of H.Y. Jin was supported by Project Funded by  NSF of China (No. 11501218). The research of  T. Xiang  was  funded by the NSF of China (No. 11601516,   11571364 and 11571363).

\end{document}